\documentclass[11pt]{amsart}
\usepackage{amssymb,mathrsfs,graphicx,extpfeil}
\usepackage{epsfig}
\usepackage{indentfirst, latexsym, amssymb, enumerate,amsmath,graphicx,amsthm}
\usepackage{float}
\usepackage{colortbl}
\usepackage{epsfig,subfigure}

\title[Discrete Cucker-Smale model with randomly switching topologies]{Emergence of stochastic flocking for the discrete Cucker-Smale model with randomly switching topologies}

\author[Dong]{Jiu-Gang Dong}
\address[Jiu-Gang Dong]{\newline Department of Mathematics, Harbin Institute of Technology,\newline Harbin 150001, P. R. China}
\email{jgdong@hit.edu.cn}

\author[Ha]{Seung-Yeal Ha}
\address[Seung-Yeal Ha]{\newline Department of Mathematical Sciences and Research Institute of Mathematics, \newline Seoul National University, Seoul 08826 \newline
and School of Mathematics, Korea Institute for Advanced Study,\newline
 Hoegiro 85, Seoul 02455, Korea (Republic of)}
\email{syha@snu.ac.kr}

\author[Jung]{Jinwook Jung}
\address[Jinwook Jung]{\newline Department of Mathematical Sciences, Seoul National University,\newline Seoul 08826, Korea (Republic of)}
\email{warp100@snu.ac.kr}

\author[Kim]{Doheon Kim}
\address[Doheon Kim]{\newline School of Mathematics, Korea Institute for Advanced Study,\newline Hoegiro 85, Seoul 02455, Korea (Republic of)}
\email{doheonkim@kias.re.kr}

\newtheorem{theorem}{Theorem}[section]
\newtheorem{lemma}{Lemma}[section]

\newtheorem{proposition}{Proposition}[section]

\newtheorem{remark}{Remark}[section]

\newtheorem{definition}{Definition}[section]

\newcommand{\bbr}{\mathbb R}

\newcommand{\bbz}{\mathbb Z}

\newcommand{\bbn} {\mathbb N}

\newcommand{\e}{\varepsilon}

\newcommand{\G}{\mathcal G}

\newcommand{\D}{\mathcal D}

\newcommand{\E}{\mathcal E}

\allowdisplaybreaks

\def\charf {\mbox{{\text 1}\kern-.30em {\text l}}}
\begin{document}
%%%%%%%%%%%%%%%%

\date{\today}


\allowdisplaybreaks

\thanks{\textbf{Acknowledgment.} The work of S.-Y. Ha is supported by the National Research Foundation of Korea(NRF-2017R1A2B2001864).  The work of J. Jung is supported by the National Research Foundation of Korea (NRF) grant funded by the Korea government (MSIP) : NRF-2016K2A9A2A13003815.}
\begin{abstract}
We study emergent dynamics of the discrete Cucker-Smale (in short, DCS) model with randomly switching network topologies. For this, we provide a sufficient framework leading to the stochastic flocking with probability one. Our sufficient framework is formulated in terms of an admissible set of network topologies realized by digraphs and probability density function for random switching times. As examples for the law of switching times, we use the Poisson process and the geometric process and show that these two processes satisfy the required conditions in a given framework so that we have a stochastic flocking with probability one. As a corollary of our flocking analysis, we improve the earlier result \cite{D-H-J-K} on the continuous C-S model.
\end{abstract}
\maketitle \centerline{\date}

%\tableofcontents

\section{Introduction} \label{sec:1}
\setcounter{equation}{0}
The jargon ``{\it flocking}" denotes a collective phenomenon in which self-propelled particles are organized in an ordered motion with the same velocity using only environmental information, and it is sometimes used interchangeably in literature with other jargons such as swarming and herding. Flocking phenomena are often observed in many biological complex systems \cite{B-C, C-H-L, M-T2, T-T, V, V2} and found in applications in the engineering domain such as sensor networks \cite{L-P}. In 2007, Cucker and Smale \cite{C-S}  introduced a simple second-order Netwon-like model in the spirit of Vicsek's model \cite{V} as an analytically manageable model, and they derived sufficient frameworks leading to the global and local asymptotic flocking on the {\it complete graph}. After their work, the Cucker-Smale(CS) model has been extensively studied in applied math community from several aspects, e.g., stochastic noises \cite{A-H, B-C-C, C-M, H-L-L}, discrete approximation \cite{D-H-K2, H-K-L}, random failures \cite{ D-M,D-M2, H-M,  R-L-X},   time-delayed communications \cite{C-H, C-L, E-H-S, L-W}, collision avoidance \cite{C-D}, network topologies \cite{C-D2,C-D3, D-Q, L-H, L-X, P-V, Shen}, mean-field limit \cite{H-K-Z, H-L},  kinetic and hydrodynamic descriptions \cite{C-C-R, C-F-R-T, F-H-T, H-T, P-S}, nonsymmetric communication networks \cite{M-T}, bonding force \cite{P-K-H}, extra internal variables \cite{H-K-K-S, H-K-R}, etc (see recent survey articles \cite{A-B-F, C-H-L}). 

In this paper, among others, we focus on the effect of the randomly switching network topologies in asymptotic flocking. To fix the idea, consider a group of migrating birds. Then, one easily observes that leaders can change over time, birds' flock can split into several local groups, and then some of them can merge to form a bigger group from time to time. Thus, the underlying network topologies keep changing over time. Hence, it is natural to consider flocking models with randomly switching network topologies for realistic flocking modeling. In this direction, authors recently studied the emergent dynamics of the continuous CS model with randomly switching network topologies and provided a sufficient framework for asymptotic global flocking in \cite{D-H-J-K} (see related works \cite{L-A-J, Wu06} for linear consensus models with randomly switching topologies).

The goal of this work is to revisit the discrete Cucker-Smale model with randomly switching network topologies, and provide a sufficient framework for global flocking and improve the earlier result \cite{D-H-J-K} on the continuous  model. To address our problem, we first discuss our problem setting. Let ${\mathcal S}_{\mathcal G} :=\{ {\mathcal G}_1, \cdots, {\mathcal G}_{N_G} \}$ be the admissible set of network topologies (digraphs) satisfying a certain connectivity condition to be specified later, and let $x_i$ and $v_i$ be the position and velocity of the $i$-th C-S particle under the influence of randomly switching network topologies. Then, their dynamics is governed by the CS model:
\begin{equation}\label{CS}
\begin{cases}
\displaystyle \dot{x}_i = v_i,  \quad t>0, \quad 1 \leq i \leq N, \\
\displaystyle \dot{v}_i = \frac{1}{N}\sum_{j=1}^N \chi_{ij}^\sigma \phi(\|x_j - x_i\|)  \left(v_j-v_i\right),
\end{cases}
\end{equation}

\noindent where $\|\cdot\|$ denotes the standard $\ell^2$-norm in $\bbr^d$ and the communication weight $\phi: [0,\infty) \to \bbr_+$ is bounded, Lipschitz continuous and monotonically decreasing:
\begin{align}
\begin{aligned} \label{comm}
& 0 < \phi(r) \le \phi(0)=: \kappa, \quad [\phi]_{Lip} := \sup_{x \not = y} \frac{|\phi(x) - \phi(y)|}{|x-y|} < \infty, \\
&  (\phi(r)-\phi(s))(r-s) \le 0, \quad r,s\geq 0.
\end{aligned}
\end{align}
The adjacent matrix $(\chi_{ij}^\sigma)$ denotes the time-dependent network topology corresponding to the randomly switching law $\sigma :[0,\infty)\times\Omega \to \{1 , \cdots, N_G\}$. The law $\sigma$ is associated with a sequence $\{t_\ell\}_{\ell\ge0}$ of random variables, called {\it random switching instants}, such that for each fixed $\omega \in \Omega$, $\sigma(\omega) :[0,\infty) \to \{1 , \cdots, N_G\}$ is piecewise constant and right-continuous whose discontinuities are $\{t_\ell(\omega)\}_{\ell\ge0}$.  In addition, the random law $\sigma$ assigns \eqref{CS} with a network topology at each instant. Specifically, once  we fix $\omega\in\Omega$ and an instant $t>0$, then $\sigma(t,\omega) = \sigma(t_\ell(\omega)) = k$ for some $1\le k \le N_G$ and $\ell \geq 0$, and we equip system \eqref{CS} with the network topology $(\chi_{ij}^{\sigma(t,\omega)})$ given by the 0-1 adjacency matrix of the $k$-th digraph $\mathcal{G}_k \in {\mathcal S}_{\mathcal G}$.\newline

Next, we consider the discrete analog of CS whose emergent dynamics we will study. The discrete counterpart of CS will be given by the forward first-order Euler method. Let $h:=\Delta t$ be the time-step, and for notational simplicity, we use the following handy notation:
\[ x_i[t]:=x_i(th), \quad  v_i[t]:=v_i(th) \quad \mbox{and} \quad \sigma[t]:=\sigma(th), \quad t \in \{0  \} \cup \bbz_+. \]
Then, the discrete C-S model reads as follows: 
\begin{equation} \label{CS_d}
\begin{cases} 
\displaystyle x_i[t+1] = x_i[t] + hv_i[t],\quad \quad t = 0, 1, \cdots, \quad 1 \le i \le N, \\
\displaystyle v_i[t+1] = v_i[t] + \frac{h}{N}\sum_{j=1}^N \chi_{ij}^{\sigma[t]} \phi(\|x_j[t]-x_i[t]\|)(v_j[t] - v_i[t]).
\end{cases}
\end{equation}
In this paper, we look for a sufficient framework leading to the asymptotic global flocking:
\begin{equation*} \label{A-1}
\sup_{0 \leq t < \infty} \max_{1 \leq i, j \leq N} \| x_i[t] - x_j[t]  \| < \infty, \quad  \lim_{t \to \infty} \max_{1 \leq i, j \leq N} \| v_i[t] - v_j[t] \| = 0.
\end{equation*}
Before we present our main result, we introduce configuration diameters:
\begin{align*}
\begin{aligned} 
& X := (x_1, \cdots, x_N), \quad V := (v_1, \cdots, v_N), \\
& {\mathcal D}(X) := \max_{1 \leq i, j \leq N}\|x_i - x_j\|, \quad {\mathcal D}(V) := \max_{1 \leq i, j \leq N} \|v_i - v_j\|.
\end{aligned}
\end{align*}
Below, we briefly discuss our main result for \eqref{CS_d}. As aforementioned, we need to impose certain conditions on the admissible set ${\mathcal S}_{\mathcal G}$ for network topologies and the probability distribution of the random switching times $\{t_\ell \}$. As discussed in \cite{D-H-J-K}, the random dwelling times between immediate switchings need to be made not too short and not too long. If this happens and choice of network topologies in the admissible set ${\mathcal S}_{\mathcal G}$ is made randomly, then the union of network topologies in ${\mathcal S}_{\mathcal G}$ will govern the asymptotic dynamics in a sufficiently long finite time interval, due to the law of large numbers. Hence, we need to assume that the union graph contains a spanning tree so that any pair of two particles will communicate eventually in a finite-time interval. These heuristic arguments can be made rigorously in Section \ref{sec:3.1}. Then under these assumptions on the network topologies and random switching times, if the time-step $h$, choice probability $p_k$ for the network topology ${\mathcal G}_k$ and communication weight $\phi$ satisfy a set of assumptions:
\[   0 < h\kappa < 1, \quad \frac{ \log\left(\frac{1}{1-h\kappa}\right)}{\min\limits_{1\leq k\leq N_G}\log\frac{1}{1-p_k}} \ll 1, \quad \frac{1}{\phi(r)}=O(r^{ \varepsilon })\quad\mbox{as}\quad r\to\infty,    \]
where $\varepsilon$ is a small positive constant to be determined later, then we can derive asymptotic flocking with probability one (see Theorem \ref{T3.1}):
\[\mathbb{P}\Big( \omega : \exists \ x^\infty>0 \mbox{ s.t } \sup_{0\le t < \infty }\D(X[t,\omega]) \le x^\infty, \mbox{ and } \lim_{t \to \infty} \D(V[t,\omega]) =0 \Big) = 1. \]

The rest of this paper is organized as follows. In Section 2, we briefly review basic terminologies and lemmas to be used throughout this paper, such as directed graphs, scrambling matrices, and their basic properties. We also provide a monotonicity lemma for velocity diameter. In Section \ref{sec:3}, we provide a framework and our main results of this paper. In Section \ref{sec:4}, we provide technical proof for our main result. In Section \ref{sec:5}, we consider two concrete stochastic processes for dwelling times (Poisson process and geometric process) and then show that they satisfy the proposed framework which results in the asymptotic flocking. We also discuss some improvements for the continuous model by removing the compact support assumption on the probability density function of switching times. Finally, Section \ref{sec:6} is devoted to a summary of our main results and some future issues to be discussed later.

\vspace{0.5cm}

\noindent {\bf Notation} Throughout the paper, $(\Omega, \mathscr{F}, \mathbb{P}$) denotes a generic probability space. For matrices $A = (a_{ij})_{N\times N}$ and $B = (b_{ij})_{N\times N}$, we denote $A \ge B$ if $a_{ij} \ge b_{ij}$ for all $i$ and $j$. Moreover, we define the Frobenius norm of $A$ as follows.
\[ \| A \|_F := \sqrt{\mbox{tr}(A^* A)}, \]
and $I_d = \mbox{diag}(1, \cdots, 1)$ is the $d \times d$ identity matrix.

\section{Preliminaries}\label{sec:2}
\setcounter{equation}{0}
In this section, we briefly review basic materials on the directed graphs and matrix algebra, and then study the dissipative structure of system \eqref{CS_d}. 

\subsection{Directed graphs and matrix algebra} \label{sec:2.1}
For the modeling of network topology, we use directed graphs (digraphs). A digraph $\mathcal{G} = (\mathcal{V}, \mathcal{E})$ consists of a vertex set $\mathcal{V} = \{1, \cdots, N\}$  and an edge set $\mathcal{E} \subset \mathcal{V} \times \mathcal{V}$ consisting of ordered connected pairs of vertices. We say $j$ is a neighbor of $i$ if $(j,i)\in \mathcal{E}$ and $\mathcal{N}_i := \{ j : (j,i)\in \mathcal{E}\}$ denotes the neighbor set of $i$. If $(i,i)\in \mathcal{E}$, then $\mathcal{G}$ is said to have a self-loop at $i$. For a given digraph $\mathcal{G} = (\mathcal{V}, \mathcal{E})$, we consider its 0-1 adjacency matrix $\chi = (\chi_{ij})$, where
\[\chi_{ij} := \left\{\begin{array}{ccc}
1 & \mbox{ if } & (j,i)\in \mathcal{E},\\
0 & \mbox{ if } & (j,i)\notin \mathcal{E}.
\end{array}\right.\]
A path in $\mathcal{G}$ from $i$ to $j$ is a sequence of distinct vertices $i_0 = i$, $i_1$, $\cdots$, $i_n = j$ such that  $(i_{m-1}, i_m) \in \mathcal{E}$ for every $1 \le m \le n$. If there is a path from $i$ to $j$, then we say $j$ is reachable from $i$. Moreover, a digraph $\mathcal{G}$ is said to have a spanning tree (or rooted) if $\mathcal{G}$ has a vertex $i$ from which any other vertices are reachable. 

Next, we review some concepts and results in matrix algebra which will be crucially used in later sections.  First, we introduce several concepts of nonnegative matrices in the following definition.
\begin{definition} \label{D2.1}
Let $A=(a_{ij})$ be a nonnegative $N \times N$ matrix.
\begin{enumerate}
\item
$A$ is a {\em stochastic} matrix, if its row-sum is equal to unity:
\[ \sum_{j=1}^{N} a_{ij} = 1, \quad 1 \leq i \leq N. \]
\item
$A$ is a {\em scrambling} matrix, if for each pair of indices $i$ and $j$, there exists an index $k$ such that $a_{ik}>0$ and $a_{jk}>0$.
\item
$A$ is a weighted {\em adjacency matrix} of a digraph $\G$ if the following holds:
\[ a_{ij}>0 \quad \Longleftrightarrow \quad (j,i)\in\E. \]
In this case, we write $\G=\G(A)$.
\item
The {\em ergodicity coefficient} of $A$  is defined as follows.
\begin{equation*} \label{B-1}
\mu(A):=\min_{i,j}\sum_{k=1}^N \min\{a_{ik},a_{jk}\}.
\end{equation*}
\end{enumerate}
\end{definition}
\begin{remark}
The following assertions easily follow from the above definition:
\begin{enumerate}
\item
$A$ is scrambling if and only if $\mu(A)>0$.
\item
For nonnegative matrices $A$ and $B$,
\[
 A\geq B \quad \Longrightarrow \quad \mu(A)\geq\mu(B).
\]
\end{enumerate}
\end{remark}

In the sequel, we present two lemmas without proofs. These lemmas will be crucially used in next sections.  

\begin{lemma}\label{L2.1}
\emph{(Lemma 2.2, \cite{D-H-K})}
Suppose that a nonnegative $N\times N$ matrix $A=(a_{ij})$ is stochastic, and let $B=(b_i^j)$, $Z=(z_i^j)$ and $W=(w_i^j)$ be $N\times d$ matrices such that
\[
 W = AZ+B.
\]
Then, we have
\begin{equation*}
\max_{i,k}\|w_i-w_k\|\leq(1-\mu(A))\max_{l,m}\|z_l-z_m\|+\sqrt{2}\|B\|_F,
\end{equation*}
where
\[
z_i:=(z_i^1,\cdots,z_i^d), \quad b_i:=(b_i^1,\cdots,b_i^d), \quad w_i:=(w_i^1,\cdots,w_i^d), i=1,\cdots,N.
\]
\end{lemma}

\begin{lemma}\label{L2.2}
\emph{(Theorem 5.1, \cite{Wu06})}
If $A_i$ are nonnegative $N\times N$ matrices with positive diagonal
elements such that $\G(A_i)$ has a spanning tree for all $1\leq i\leq N-1$, then $A_1A_2\ldots
A_{N-1}$ is scrambling.
\end{lemma}
\vspace{0.2cm}

\subsection{Matrix formulation of the DCS model} \label{sec:2.2} In this subsection, we provide a priori estimate for the velocity diameter along the sample path. For convenience, we suppress $\omega$-dependence in the sequel. First, we consider the R.H.S. of $\eqref{CS_d}_2$:
\begin{align}
\begin{aligned} \label{B-1}
\mbox{RHS}_2 &:= v_i[t] + \frac{h}{N}\sum_{j=1}^N \chi_{ij}^{\sigma[t]} \phi(\|x_j[t]-x_i[t]\|)(v_j[t] - v_i[t]) \\
&= \Big( 1 - \frac{h}{N} \sum_{j=1}^N \chi_{ij}^{\sigma[t]} \phi( \| x_j[t] - x_i[t] \|) + \frac{h}{N} \chi_{ii}^{\sigma[t]} \phi(\| x_j[t] - x_i[t] \|) \Big) v_i[t]  \\
&\hspace{0.2cm} + \frac{h}{N} \sum_{j \not = i} \chi_{ij}^{\sigma[t]} \phi( \| x_j[t] - x_i[t] \|) v_j[t].
\end{aligned}
\end{align}
To rewrite \eqref{B-1} as a matrix form, we introduce $N\times N$ matrix $ {L_k}[t]$ $(k=1,\cdots,N_G)$:
\[
  L_{k }[t]:=  D_{k }[t]-  A_{k }[t],
\]
where the matrices $A_{k }[t]$ and $D_{k }(t)$ are defined as follows.
\begin{align}
\begin{aligned} \label{B-2}
A_{k }[t] &=\left(  a_{ij}^{k}[t]\right), \quad D_{k }(t)=\mbox{diag}\left(  d_1^{k}[t],\cdots, d_N^{k }[t]\right), \\
a_{ij}^{k }[t] &:=\chi_{ij}^{k}\phi(\|x_i[t]-x_j[t]\|)\quad\mbox{and}\quad   d_i^{k }[t]=\sum_{j=1}^N \chi_{ij}^{k}\phi(\|x_i[t]-x_j[t]\|).
\end{aligned}
\end{align}
Finally, we combine \eqref{B-1} and \eqref{B-2} to rewrite system \eqref{CS_d} as a matrix form:
\begin{align}
\begin{cases} \label{B-4}
\displaystyle X[t+1] = X[t] + hV[t], \quad t = 0, 1, \cdots, \\
\displaystyle V[t+1] = \left(Id - \frac{h}{N} L_{\sigma[t]}[t]\right)V[t],
\end{cases}
\end{align}
and define
\begin{equation}\label{B-5}
\Phi[s_1, s_0,\omega] := \prod_{t=s_0}^{s_1-1} \left(Id - \frac{h}{N}L_{\sigma[t,\omega]}\right).
\end{equation}
Next, we study the monotonicity of $\D(V[t])$.
\begin{lemma}\label{L2.3}
Suppose that the time-step and coupling strength satisfy
 \[0 < h\kappa <1, \]
 and let $(X[t], V[t])$ be a solution to \eqref{B-4}. Then for each fixed $\omega\in\Omega$, we have
\[
\D(V[t+1]) \le \D(V[t]), \quad \D(X[t+1]) \le \D(X[t]) + h \D(V[t]),
\]
for any $t \in \bbn \cup \{0\}$.
\end{lemma}
\begin{proof}
First, we use 
\[ \phi(r) \leq \phi(0) = \kappa \quad \mbox{in \eqref{comm} and} \quad  h\kappa <1 \]
to see that for each $i=1, \cdots ,N$ and $k=1,\cdots, N_G$,
\[\frac{h}{N}d_i^k[t] = \frac{h}{N}\sum_{j=1}^N \chi_{ij}^k \phi(\|x_i[t]-x_j[t]\|) \leq h \kappa  <1. \]
This implies that $Id - \frac{h}{N} L_{\sigma[t]}[t]$ is always nonnegative and its low sum is one, i.e., stochastic. Since the ergodicity coefficient of any nonnegative matrices is nonnegative, we apply Lemma \ref{L2.1} to get
\[ \D(V[t+1]) \le \left(1-\mu\left(Id - \frac{h}{N} L_{\sigma[t]}[t]\right)\right) \D(V[t]) \le \D(V[t]).\]
The other estimate for the spatial diameter readily follows from $\eqref{B-4}_1$.
\end{proof}

\vspace{0.4cm}

\section{A framework and main result} \label{sec:3}
\setcounter{equation}{0}
In this section, we briefly present a framework for stochastic flocking in terms of admissible network topologies and switching times, and state our main result.

\subsection{Problem setup}  \label{sec:3.1}
Let  $(\Omega, \mathscr{F}, \mathbb{P}$) be a probability space. Then, we set the sequence of (possibly) switching times $\{t_\ell\}_{\ell \ge 0}$ to be a sequence of $\bbn$-valued independent random variables on $(\Omega, \mathscr{F}, \mathbb{P}$). Then, we assume that the switching law $\sigma : \bbn \times \Omega \to \{1, \cdots, N_G\}$ and the sequence $\{t_{\ell+1}-t_\ell \}_{\ell\ge 0}$ of dwelling times satisfy the following conditions: 
\begin{itemize}
\item
For each $\ell \ge 0$, $t_{\ell+1}-t_\ell$ is given by
\begin{equation} \label{C-1} 
t_{\ell+1}-t_\ell = 1 + T_{\ell}, 
\end{equation}
where $\{T_{\ell}\}_{\ell\geq0}$ is a given sequence of independent integer-valued nonnegative random variables.
%\[ t_{\ell+1}-t_\ell = 1 + N_{\lambda_\ell}, \]
%where $N_{\lambda}$ denotes a Poisson random variable with parameter $\lambda >0$.

\vspace{0.2cm}

\item For each $\ell \geq 0$ and $\omega \in \Omega$, $\sigma_t(\omega):= \sigma(t, \omega)$ is constant on the interval $t\in[t_\ell(\omega),t_{\ell+1}(\omega))$.

\vspace{0.2cm}

\item $\{\sigma_{t_{\ell}} \}_{\ell\ge 0}$ is a sequence of i.i.d. random variables such that for any $\ell \geq 0$,
\[
\mathbb P(\sigma_{t_\ell}= k)=p_k,\quad \mbox{ for each } k=1,\cdots,N_G,
\]
where $p_1,\cdots,p_{N_G}>0$ are given positive constants satisfying $p_1+\cdots+p_{N_G}=1$.
\end{itemize}

\vspace{0.2cm}

We briefly comment on the random switching time process $\{t_\ell \}$ as follows. Suppose that we have process at $t = t_\ell$. Then, the next switching time $t_{\ell + 1}$ is determined by the relation \eqref{C-1} via the integer-valued and nonnegative random process $T_\ell$:
\[ t_{\ell + 1} = t_\ell + 1 + T_\ell, \]
In Section \ref{sec:5}, we will consider two explicit processes (Poisson and geometric processes) and show that they satisfy our framework that will be presented soon.  \newline

For each $k=1, \cdots, N_G$, let $\mathcal G_k=(\mathcal V, \mathcal E_k)$ be the $k$-th admissible digraph. Then, for each $t\geq0$ and $\omega\in\Omega$, the time-dependent network topology $(\chi_{ij}^\sigma)=(\chi_{ij}^{\sigma[t,\omega]})$ is determined by

\[\chi_{ij}^{\sigma[t,\omega]} := \left\{\begin{array}{ccc}
1 & \mbox{ if } & (j,i)\in\mathcal E_{\sigma[t,\omega]},\\
0 & \mbox{ if } & (j,i)\notin \mathcal E_{\sigma[t,\omega]}.
\end{array}\right.\]
\vspace{0.1cm}

\noindent For technical reasons and without loss of generality, we assume that each admissible digraph $\mathcal{G}_k$ has a self-loop at each vertex, and we define the union graph of $\mathcal{G}_{\sigma[t,\omega]}$ for $s_0\le t < s_1$ and $\omega\in\Omega$:
\[\mathcal{G}([s_0, s_1))(\omega) := \bigcup_{t=s_0}^{s_1-1}\mathcal{G}_{\sigma[t,\omega]} =\left(\mathcal{V}, \bigcup_{t= s_0}^{s_1-1}\mathcal{E}_{\sigma[t,\omega]}\right).\\ \]
Note that the network topology might remain unchanged at switching instants. In other words, it would be possible that $\sigma_{t_{\ell+1}}(\omega)=\sigma_{t_{\ell}}(\omega)$ for some $\ell\geq0$ and $\omega \in \Omega$.

\subsection{A framework and main result} \label{sec:3.2}
In this subsection, we present a sufficient framework in terms of the structure of admissible network topologies and laws of switching times and our main result. \newline

For each positive integer $n$ and positive real number $c>0$, we define an increasing sequence $\{a_\ell(n,c)\}_{\ell\in\mathbb{N}}$ of integers by the following recursive relation:
\begin{equation}\label{C-2}
a_0(n,c)=0,\qquad a_{\ell+1}(n,c)=a_\ell(n,c)+n+\lfloor c\log (\ell+1)\rfloor, \qquad (\ell\in\mathbb{N}).
\end{equation}
In what follows, we sometimes suppress $n$ and $c$ dependence for $a_\ell(n,c)$:
\[ a_\ell := a_\ell(n,c). \]
We impose the following assumptions $(\mathcal A)$ on the set  ${\mathcal S}_{\mathcal G}$ of admissible digraphs and the probability density function of the switching times:
\begin{itemize}
\item $(\mathcal A1)$: The union digraph from ${\mathcal S}_{\mathcal G}$ has a spanning tree:
\[
\bigcup_{1\leq k\leq N_G}\mathcal{G}_k =\left(\mathcal{V}, \bigcup_{1\leq k\leq N_G}\mathcal{E}_{k}\right)\quad\mbox{has a spanning tree.}
\]
\item 
$(\mathcal A2)$:  The dwelling times are bounded in probability: for some $M\in\bbn$,
\[\mathbb{P}\left(  \omega \ : \ \sum_{\ell=a_{(i-1)(N-1)}}^{a_{i(N-1)}-1} T_\ell(\omega) \ge M(n + \lfloor c \log i(N-1)\rfloor), \quad \mbox{for some } i \in \bbn\right )  \le \tilde{p}(n),  \]
where $\tilde{p}(n) \to 0$ as $n \to \infty$.
\end{itemize}

Now, we are ready to state our main theorem on the emergence of flocking.
\begin{theorem}\label{T3.1}
Suppose that parameters $N$, $h$, $\kappa$, $p_k$'s and communication weight $\phi$ satisfy the following conditions:
\[  0 < h\kappa < 1, \quad \frac{(M+N-1) \log\left(\frac{1}{1-h\kappa}\right)}{\min\limits_{1\leq k\leq N_G}\log\frac{1}{1-p_k}}<1, \quad \frac{1}{\phi(r)}=O(r^{ \varepsilon })\quad\mbox{as}\quad r\to\infty,
\]
where $M$ is a positive constant in $({\mathcal A}2)$ and $\varepsilon$ is a positive constant satisfying
\[  0\leq\varepsilon<\frac{1}{N-1}-\frac{(M+N-1) \log\left(\frac{1}{1-h\kappa}\right) }{(N-1)\min\limits_{1\leq k\leq N_G}\log\frac{1}{1-p_k}},
\]
and let $(X,V)$ be a solution process to \eqref{CS_d}. Then, asymptotic global flocking is exhibited with probability one:
\[\mathbb{P}\Big( \omega : \exists \ x^\infty>0 \mbox{ s.t } \sup_{0\le t < \infty }\D(X[t,\omega]) \le x^\infty, \mbox{ and } \lim_{t \to \infty} \D(V[t,\omega]) =0 \Big) = 1. \]
\end{theorem}
\begin{proof}
The proof will be given in next section.
\end{proof}

\section{Emergence of stochastic flocking} \label{sec:4}
\setcounter{equation}{0}
In this section, we provide a sufficient framework for the emergence of flocking in system \eqref{CS_d}, and provide a proof for Theorem \ref{T3.1}. \newline

First, we present a priori assumptions ($\tilde{\mathcal A}$1) - ($\tilde{\mathcal A}$3): for a fixed $\omega\in\Omega$, 
\begin{itemize}
\item
($\tilde{\mathcal A}$1):~there exist $n\in\bbn$, $n>0$ and  $c>0$  such that 
\[ cN\log\left(\frac{1}{1-h\kappa}\right)<1, \]
and the subsequence $\{t_\ell^*\}_{\ell\in\bbn}\subset\{t_\ell\}_{\ell\in\bbn}$ defined by $t_\ell^*:= t_{a_\ell(n,c)}$ in \eqref{C-2} satisfies
\[
 \mathcal G([t_\ell^*,t_{\ell+1}^*))(\omega)~\mbox{has a spanning tree for all $\ell\geq0$}.
\]
\vspace{0.1cm}
\item
($\tilde{\mathcal A}$2):~there exist $n\in\bbn$, $M\in\bbn$, and  $c>0$  such that the random sequence $\{T_{\ell}(\omega)\}_{\ell\ge0}$ satisfy the following boundedness condition:
\begin{equation} \label{D-0}
\sum_{\ell=a_{(i-1)(N-1)}}^{a_{i(N-1)}-1}T_{\ell}(\omega) < M( n + \lfloor c \log i(N-1)\rfloor), \quad \mbox{for each } i \in \bbn. 
\end{equation}
\vspace{0.1cm}
\item
($\tilde{\mathcal A}$3):~the position diameter is uniformly bounded in time:
\begin{equation*} \label{D-5}
\sup_{0\leq t<\infty}\D(X[t,\omega])\leq x^{\infty} < \infty.
\end{equation*}
\end{itemize}
For the proof of Theorem \ref{T3.1}, we split its proof into three steps as follows: \newline
\begin{itemize}
\item
Step A:~A priori assumptions ($\tilde{\mathcal A}$1), ($\tilde{\mathcal A}$2) and ($\tilde{\mathcal A}$3) imply stochastic flocking.
\vspace{0.1cm}
\item
Step B:~Conditions on initial configuration implies the a priori assumption ($\tilde{\mathcal A}$3).
\vspace{0.1cm}
\item
Step C:~Framework $({\mathcal A}1)$-$({\mathcal A}2)$ imply condition on initial configuration and a priori assumptions  ($\tilde{\mathcal A}$1)-($\tilde{\mathcal A}$2).
\end{itemize}

\vspace{0.5cm}

In the following three subsections, we present each step one by one.
\subsection{Step A (A priori assumptions imply stochastic flocking)} \label{sec:4.1} 
In this subsection, we show that the a priori assumptions ($\tilde{\mathcal A}$1), ($\tilde{\mathcal A}$2) and ($\tilde{\mathcal A}$3) imply stochastic flocking. 
\begin{lemma}\label{L4.1}
Suppose that a priori assumptions $(\tilde{\mathcal A}1)$ and $(\tilde{\mathcal A}3)$ hold, and system parameters satisfy 
\[ 0 < h \kappa < 1, \]
and let $(X, V)$ be a solution to  \eqref{CS_d}. Then,  the matrix $\Phi[t_{r(N-1)}^*, t_{(r-1)(N-1)}^*]$ defined in \eqref{B-5} satisfies 
\begin{enumerate}
\item
$\Phi[t_{r(N-1)}^*, t_{(r-1)(N-1)}^*]$ is stochastic for each $r \in \bbn$.
\vspace{0.2cm}
\item
The ergodicity coefficient of $ \Phi[t_{r(N-1)}^*, t_{(r-1)(N-1)}^*]$ satisfies
\[\mu(\Phi[t_{r(N-1)}^*, t_{(r-1)(N-1)}^*]) \ge\left(1-h\kappa \right)^{t_{r(N-1)}^* - t_{(r-1)(N-1)}^*} \left(\frac{h\phi(x^\infty)}{N(1-h\kappa)}\right)^{N-1}, \]
for each $r \in \bbn$.
\end{enumerate}
\end{lemma}
\begin{proof}
For each $r \in \bbn$ and $1 \le i \le N-1$, we consider $\Phi[t_{(r-1)(N-1)+i}^*, t_{(r-1)(N-1)+i-1}^*]$. Here, we assume that
\[\{t_\ell\}_{\ell \ge 0} \cap [t_{(r-1)(N-1)+i-1}^*, t_{(r-1)(N-1)+i}^*] = \{ t_{\ell_1}, \cdots, t_{\ell_{q+1}}\}. \]
Then, we get
\[ \Phi[t_{(r-1)(N-1)+i}^*, t_{(r-1)(N-1)+i-1}^*] = \prod_{j=1}^q \Phi[t_{\ell_{j+1}}, t_{\ell_j}],
\]
and for each $j$, we can choose $k_j \in \{1, \cdots, N_G\}$ such that 
\[ \sigma[t, \omega] = k_j \quad \mbox{for all}~~t \in \bbn \cap [t_{\ell_j}, t_{\ell_{j+1}}). \]
Thus, we have
\begin{align*}
\Phi[t_{\ell_{j+1}}, t_{\ell_j}] &= \prod_{t=t_{\ell_j}}^{t_{\ell_{j+1}}-1} \left(Id - \frac{h}{N} L_{k_j} [t]\right) = \prod_{t=t_{\ell_j}}^{t_{\ell_{j+1}}-1} \left(Id - \frac{h}{N} D_{k_j} [t] + \frac{h}{N}A_{k_j}[t]\right)\\
&\ge  \prod_{t=t_{\ell_j}}^{t_{\ell_{j+1}}-1} \left((1-h\kappa)Id + \frac{h}{N}A_{k_j}[t]\right),
\end{align*}
and consequently,
\begin{align*}
&\Phi[t_{(r-1)(N-1)+i}^*, t_{(r-1)(N-1)+i-1}^*] = \prod_{j=1}^q \Phi[t_{\ell_{j+1}}, t_{\ell_j}] \ge \prod_{j=1}^q \prod_{t=t_{\ell_j}}^{t_{\ell_{j+1}}-1} \left( (1-h\kappa)Id + \frac{h}{N}A_{k_j}[t]\right ) \\
& \hspace{1cm}  \ge \left(1-h\kappa\right)^{t_{(r-1)(N-1)+i}^* - t_{(r-1)(N-1)+i-1}^*-1} \frac{h}{N} \sum_{j=1}^q \sum_{t=t_{\ell_j}}^{t_{\ell_{j+1}}-1} A_{k_j}[t]\\
& \hspace{1cm}  \ge \left(1-h\kappa\right)^{t_{(r-1)(N-1)+i}^* - t_{(r-1)(N-1)+i-1}^*} \frac{h\phi(x^\infty)}{N(1-h\kappa)} F_i,
\end{align*}
where $F_i$ denotes the $0$-$1$ adjacency matrix of the graph $\mathcal{G}[t_{(r-1)(N-1)+i-1}^*, t_{(r-1)(N-1)+i}^*)$  and the second inequality follows from the fact:
\[
\begin{split}
\prod_{i=1}^m (a Id + B_i) &= a^m Id + a^{m-1}\sum_{i=1}^m B_i + \cdots \ge a^{m-1}\sum_{i=1}^m B_i, \quad \mbox{for all } a>0, \ B_i \ge 0.
\end{split}\]
 Hence, we have
\begin{align*}
\Phi[t_{r(N-1)}^*, t_{(r-1)(N-1)}^*]&=\prod_{i=1}^{N-1}\Phi[t_{(r-1)(N-1)+i}^*, t_{(r-1)(N-1)+i-1}^*]\\
&\ge \left(1-h\kappa \right)^{t_{r(N-1)}^* - t_{(r-1)(N-1)}^*} \left(\frac{h\phi(x^\infty)}{N(1-h\kappa)}\right)^{N-1} \Big( \prod_{i=1}^{N-1}F_i \Big).
\end{align*}
Since each $F_i$ is the 0-1 adjacency matrix of a graph with a spanning tree, $\prod_{i=1}^{N-1}F_i$ is scrambling by Lemma \ref{L2.2}. Moreover, we have
\[\mu\left(\prod_{i=1}^{N-1}F_i\right) \ge 1. \]
Thus, we can deduce the result (2). For the stochasticity, note that $Id - \frac{h}{N} L_{k} [t]$ is stochastic for any $k = 1, \cdots, N_G$. This gives the stochasticity of $\Phi[t_{\ell_{j+1}}, t_{\ell_j}]$ and consequently, we obtain the desired result.
\end{proof}
Next, we provide the decay estimates for the velocity diameter based on a priori assumptions. In the sequel, we set 
\[ X^0 := X[0], \quad V^0 := V[0]. \]
\begin{proposition}\label{P4.1}
\emph{(Velocity alignment)}
Suppose that $(\tilde{\mathcal A}1)$-$(\tilde{\mathcal A}3)$ hold, and system parameters satisfy 
\[ 0 < h \kappa < 1, \]
and let $(X, V)$ be a solution to  \eqref{CS_d}. Then for each $t \in \bbn \cap [t_{r(N-1)}^*, t_{(r+1)(N-1)}^*)$,
\begin{align*}
\begin{aligned}
&\D(V[t]) \le  \D(V^0) \\
& \hspace{0.2cm} \times \exp\left[ -\left(1-h\kappa\right)^{(M+N-1) (n+c\log(N-1))}\left(\frac{h\phi(x^\infty)}{N(1-h\kappa)}\right)^{N-1} \frac{(r+1)^{1+c (M+N-1) \log(1-h\kappa)}-1}{1+c (M+N-1) \log(1-h\kappa)} \right ].
\end{aligned}
\end{align*}
\end{proposition}
\begin{proof}
For each $r \ge 0$, we use Lemma \ref{L4.1} to get
\begin{align*}
\D(V[t]) &\le \D(V[t_{r(N-1)}^*])\\
&\le \left(1-\mu\left(\Phi[t_{r(N-1)}^*, t_{(r-1)(N-1)}^*]\right)\right) \D(V[t_{(r-1)(N-1)}^*])\\
&\le \prod_{i=1}^r \left(1-\mu\left(\Phi[t_{i(N-1)}^*, t_{(i-1)(N-1)}^*]\right)\right) \D(V^0)\\
&\le \prod_{i=1}^r \left(1-\left( 1-h\kappa \right)^{t_{r(N-1)}^* - t_{(r-1)(N-1)}^*} \left(\frac{h\phi(x^\infty)}{N(1-h\kappa)}\right )^{N-1}\right) \D(V^0).
\end{align*}
Note that
\begin{align*}
t_{i(N-1)}^*  - t_{(i-1)(N-1)}^* & = \sum_{\ell=a_{(i-1)(N-1)}}^{a_{i(N-1)}-1} (t_{\ell+1} - t_\ell) =a_{i(N-1)} - a_{(i-1)(N-1)} + \sum_{\ell=a_{(i-1)(N-1)}}^{a_{i(N-1)}-1} T_{\ell} \\
&\le (M+N-1)(n+\lfloor c \log i(N-1)\rfloor).
\end{align*}
This follows from the definition of the sequence $a_\ell$ and the a priori assumption. Thus, we have
\begin{align*}
&\D(V[t]) \le \prod_{i=1}^r \left[ 1-\left(1-h\kappa \right)^{(M+N-1)(n+\lfloor c\log i(N-1)\rfloor)} \left(\frac{h\phi(x^\infty)}{N(1-h\kappa)}\right)^{N-1}\right ]\D(V^0)\\
&\le  \prod_{i=1}^r \exp\left[ -\left(1-h\kappa \right)^{(M+N-1)(n+ c\log i(N-1))} \left(\frac{h\phi(x^\infty)}{N(1-h\kappa)}\right)^{N-1}\right] \D(V^0)\\
&= \exp\left[ -\left(1-h\kappa\right)^{(M+N-1)(n+c\log(N-1))}\left(\frac{h\phi(x^\infty)}{N(1-h\kappa)}\right)^{N-1} \sum_{i=1}^r i^{c(M+N-1)\log(1-h\kappa)}\right ] \D(V^0)\\
&\le \exp\left[ -\left(1-h\kappa\right)^{(M+N-1)(n+c\log(N-1))}\left(\frac{h\phi(x^\infty)}{N(1-h\kappa)}\right)^{N-1}\int_1^{r+1} x^{c (M+N-1) \log(1-h\kappa)} dx \right ] \D(V^0)\\
&= \exp\left[ -\left(1-h\kappa\right)^{(M+N-1)(n+c\log(N-1))}\left(\frac{h\phi(x^\infty)}{N(1-h\kappa)}\right)^{N-1} \frac{(r+1)^{1+c (M+N-1)\log(1-h\kappa)}-1}{1+c (M+N-1)\log(1-h\kappa)} \right ] \D(V^0).
\end{align*}
This leads to the desired estimate.
\end{proof}
\vspace{0.3cm}

\subsection{Step B (Pathwisely, $(\tilde{\mathcal A}1)$ and $(\tilde{\mathcal A}2)$ imply flocking)} \label{sec:4.2}
In this subsection, we show that the spatial diameter is uniformly bounded. For this, we present a simple lemma without a proof.
\begin{lemma}\emph{\cite{D-H-J-K}}\label{L4.2}
For any $x>0$ and $\delta>0$, we have the following inequality:
\[
e^{-x}\leq \bigg(\frac{\delta}{e}\bigg)^\delta x^{-\delta}.
\]
\end{lemma}
%\begin{proof}
%By differentiation, we can check that the function $x\mapsto -x+\delta\log x$ attains its maximal value at $x=\delta$. Hence
%\[
%-x+\delta\log x\leq -\delta+\delta\log\delta,\quad x>0,~\delta>0.
%\]
%We take the exponential of both sides to get
%\[
%e^{-x}x^\delta\leq e^{-\delta}\delta^\delta \quad\Longrightarrow\quad e^{-x}\leq \bigg(\frac{\delta}{e}\bigg)^\delta x^{-\delta}.
%\]
%\end{proof}
Now, we show that the a priori assumption $(\tilde{\mathcal A}3)$ on the position diameter can be replaced by the assumption on initial data.
\begin{lemma}\label{L4.3}
\emph{(Pathwise uniform bound for spatial diameter)}
Suppose that $\omega \in \Omega$ satisfies a priori assumptions $(\tilde{\mathcal A}1)$-$(\tilde{\mathcal A}2)$, and system parameters and initial data satisfy  
{\small
\begin{eqnarray*}
&& (i)~ 0 < h \kappa < 1. \cr
&& (ii)~\exists~ \delta>0~~\mbox{and}~~x^\infty>0~\mbox{independent of a sample point such that} \cr
&& \hspace{1cm} \D(X^0) +\D(V^0)h (M+N-1) (n+c\log ( N-1))  \\
&& \hspace{2cm}  +~h\D(V^0)\sum_{r=1}^\infty\Bigg[ (M+N-1)(n+c\log ((r+1)(N-1)))\bigg(\frac{\delta}{e }\bigg)^{\delta}\\
&& \hspace{1cm} \times \left(\left(1-h\kappa\right)^{(M+N-1)(n+c\log(N-1))}\left(\frac{h\phi(x^\infty)}{N(1-h\kappa)}\right)^{N-1} \frac{(r+1)^{1+c(M+N-1)\log(1-h\kappa)}-1}{1+c (M+N-1) \log(1-h\kappa)}\right)^{-\delta}\Bigg] \\
&& \hspace{1cm} < x^{\infty},
\end{eqnarray*}}
and let $(X,V)$ be a solution process to \eqref{CS_d}. Then, a priori condition $(\tilde{{\mathcal A}}3)$ is fulfilled:
\[  \sup_{0 \leq t < \infty} \mathcal D(X[t,\omega])<x^\infty. \]
\end{lemma}
\begin{proof}
%For our desired result, we argue by contradiction. First, we define a set ${\mathcal B}$ and its supremum as follows:
%\[
%{\mathcal B}:=\Big\{t\ge 0:\quad \max_{0 \leq t \leq T} \D(X[t])<x^{\infty} \Big\}, \quad \tilde{t}:=\sup {\mathcal B}.
%\]
%By assumption \eqref{C-4}, the set ${\mathcal B}$ is nonempty. Now, we claim:
%\[ \sup {\mathcal B}=\infty. \]
%Suppose not, i.e. $\tilde{t}:=\sup {\mathcal B} <\infty$. Then,  we have
%\begin{equation} \label{C-101}
%\D(X(T^*))=x^{\infty}.
%\end{equation}
We use Lemma \ref{L4.1}, Lemma \ref{L4.2} and the condition (ii) to see that for each $t>0$,
\begin{align*}
&\D(X[t]) \leq \D(X^0)+h \sum_{s=0}^{t-1} \D(V[s])  \\
& \hspace{1.5cm} \leq \D(X^0)+h\D(V^0)\sum_{r=0}^\infty\Bigg[ (t^*_{(r+1)(N-1)}-t^*_{r(N-1)})  \\
& \quad\times\exp\left(-\left(1-h\kappa\right)^{(M+N-1)(n+c\log(N-1))}\left(\frac{h\phi(x^\infty)}{N(1-h\kappa)}\right)^{N-1} \frac{(r+1)^{1+c (M+N-1) \log(1-h\kappa)}-1}{1+c (M+N-1) \log(1-h\kappa)} \right)\Bigg] \\
& \hspace{1.5cm} \leq\D(X^0)+h\D(V^0)\sum_{r=0}^\infty\Bigg[ (M+N-1)(n+c\log ((r+1)(N-1)))  \\
& \quad\times\exp\left(-\left(1-h\kappa\right)^{(M+N-1)(n+c\log(N-1))}\left(\frac{h\phi(x^\infty)}{N(1-h\kappa)}\right)^{N-1} \frac{(r+1)^{1+c (M+N-1) \log(1-h\kappa)}-1}{1+c (M+N-1) \log(1-h\kappa)} \right)\Bigg] \\
& \hspace{1.5cm}  \leq\D(X^0) +h\D(V^0) (M+N-1)(n+c\log ( N-1))   \\
&\quad +h\D(V^0)\sum_{r=1}^\infty\Bigg[ (M+N-1)(n+c\log ((r+1)(N-1)))\bigg(\frac{\delta}{e }\bigg)^{\delta}\\
&\qquad\times \left(\left(1-h\kappa\right)^{(M+N-1)(n+c\log(N-1))}\left(\frac{h\phi(x^\infty)}{N(1-h\kappa)}\right)^{N-1} \frac{(r+1)^{1+c (M+N-1) \log(1-h\kappa)}-1}{1+c (M+N-1)\log(1-h\kappa)}\right)^{-\delta}\Bigg] \\
& \hspace{1.5cm}  < x^{\infty}.
\end{align*}
This implies the desired result.
\end{proof}
Finally, we combine Proposition \ref{P4.1} with ($\tilde{\mathcal A}1$)-($\tilde{\mathcal A}2$) and a condition on communication weight to show the emergence of flocking for any initial configuration.
\begin{proposition}\label{P4.2}
Suppose that $\omega \in \Omega$ satisfies a priori assumptions $(\tilde{\mathcal A}1)$-$(\tilde{\mathcal A}2)$, and system parameters and communication weight satisfy
\[ 0 < h\kappa <1 \quad  \mbox{and} \quad \frac{1}{\phi(r)}={\mathcal O}(r^{ \varepsilon }) \quad \mbox{as $r \to\infty$}, \]
where $\varepsilon$ is a positive constant satisfying the relation:
\[ 0\leq\varepsilon<\frac{1+ c (M+N-1) \log(1-h\kappa)}{N-1}, \]
and let $(X,V)$ be a solution process to \eqref{CS_d}. Then, the global flocking emerges for any initial configuration pathwise: there exists $x^\infty >0$ such that
\[
\sup_{0\leq t<\infty}\mathcal D(X(t,\omega))\leq x^\infty \quad \mbox{and}\quad \lim_{t\to\infty}\mathcal D(V(t,\omega))=0.
\]
\end{proposition}
\begin{proof} First, we choose $\delta>0$ such that
\begin{equation}\label{D-1}
\frac{1}{1+c (M+N-1) \log(1-h\kappa)}<\delta<\frac{1}{(N-1)\varepsilon}.
\end{equation}
Then, the L.H.S. in \eqref{D-1} yields the convergence of the series appearing in (ii) of Lemma \ref{L4.3}:
\begin{align*}
&\sum_{r=1}^\infty\Bigg[ (n+c\log ((r+1)(N-1)))  \left(\frac{(r+1)^{1+c (M+N-1) \log(1-h\kappa) } -1 }{1+c (M+N-1) \log(1-h\kappa)  }\right)^{-\delta}\Bigg ]\\
&\quad  \le C \sum_{r=1}^\infty \frac{ 1+ \log(r+1)}{\left((r+1)^{1+c (M+N-1) \log(1-h\kappa) } -1\right)^\delta }\\
&\quad = C\left(\sum_{r=1}^{r^*-1} + \sum_{r=r^*}^\infty\right)\frac{ 1+ \log(r+1)}{\left((r+1)^{1+c (M+N-1) \log(1-h\kappa) } -1\right)^\delta }\\
&\quad \le C \left[\sum_{r=1}^{r^*-1} \frac{ 1+ \log(r+1)}{\left((r+1)^{1+c (M+N-1) \log(1-h\kappa) } -1\right)^\delta }+\sum_{r=r^*}^\infty \frac{1 +\log(r+1)}{r^{(1+c (M+N-1) \log(1-h\kappa))\delta } }\right]\\
&\quad <\infty,
\end{align*}
where $C>0$ is a positive constant and $r^*>0$ is a natural number satisfying
\[
(r+1)^{1+c (M+N-1) \log(1-h\kappa)} -1 \ge \frac{1}{2} r^{1+c (M+N-1) \log(1-h\kappa)}, \quad a \in (0,1), \quad  ^\forall r \ge r^*.
\]
Note that the existence of $r^*$ is guaranteed since
\[
\lim_{r \to \infty}\frac{(r+1)^{1+c (M+N-1) \log(1-h\kappa)} -1}{r^{1+c (M+N-1) \log(1-h\kappa)}} = 1.
\]
\vspace{0.1cm}

\noindent Furthermore, since the series in (ii) of Lemma \ref{L4.3} converges, we may deduce that

\begin{align*}
&\D(X^0) +h\D(V^0) (M+N-1) (n+c\log ( N-1))  \\
&\quad  +~h\D(V^0)\sum_{r=1}^\infty\Bigg[ (M+N-1)(n+c\log ((r+1)(N-1)))\bigg(\frac{\delta}{e }\bigg)^{\delta}\\
&\quad \times \hspace{-0.1cm}\left(\left(1-h\kappa\right)^{(M+N-1)(n+c\log(N-1))}\left(\frac{h\phi(x^\infty)}{N(1-h\kappa)}\right)^{N-1} \frac{(r+1)^{1+c(M+N-1)\log(1-h\kappa)}-1}{1+c (M+N-1) \log(1-h\kappa)}\right)^{-\delta}\Bigg]\\
&\quad \le C\left(1+ (\phi(x^\infty))^{-(N-1)\delta}\right),
\end{align*} 
where $C$ is a constant independent of $x^\infty$. Note that the R.H.S. in \eqref{D-1} gives

\[
\phi(r)^{-(N-1)\delta}=O(r^{ (N-1)\delta \varepsilon }) \quad \mbox{as $r \to\infty$}.
\]
This implies

\begin{equation}\label{D-2}
\lim_{r \to\infty}\frac{\phi(r)^{-(N-1)\delta}}{r}=0.
\end{equation}
Thus, it follows from \eqref{D-2} that we can deduce the existence of $x^\infty$ satisfying the relation (ii) Lemma \ref{L4.3} for $\delta$ given in \eqref{D-1}. Since the condition (ii) is now satisfied, we obtain our desired results  from Lemma \ref{L4.2} and Proposition \ref{P4.1}.
\end{proof}

\vspace{0.2cm}

\subsection{Step C (Framework $({\mathcal A}1) - ({\mathcal A}2)$ implies stochastic flocking)} \label{sec:4.3}
In this subsection, we show that a priori assumptions $({\tilde {\mathcal A}}1) - ({\tilde {\mathcal A}}2)$  can be fulfilled under our framework and prove that flocking emerges with probability one. First, we show that the first a priori assumption $(\tilde{\mathcal A}1)$ can be satisfied almost surely with the appropriate choices of parameters. Below, we define a subsequence $\{t_\ell^*\}_{\ell\geq0}\subset\{t_\ell\}_{\ell\geq0}$ by $t_\ell^*:= t_{a_\ell(n,c)}$.

\begin{lemma}\label{L4.4}
Suppose that a positive integer $n$ and a real number $c>0$ are sufficiently large enough such that 
\[
\sum_{k=1}^{N_G}(1-p_k)^{n}\leq\frac{1}{2} \quad\mbox{and}\quad c>-\frac{1}{\log (1-p_k)},~~k=1,\cdots,N_G,
\]
and let $(X,V)$ be a solution process to \eqref{CS}. Then we have the following assertions.
\begin{enumerate}
\item
For each $k=1,\cdots,N_G$,  
\[ \sum\limits_{\ell=0}^\infty  (1-p_k)^{\lfloor c\log (\ell+1)\rfloor} < \infty. \]
\item

$\mathbb{P}(\omega : \mathcal{G}([t_\ell^*, t_{\ell+1}^*))(\omega) \mbox{ has a spanning tree for any } \ell \ge 0)$
\vspace{-0.1cm}
\[
 \hspace{0.5cm} \geq \exp\left( -(2\log 2) \sum_{k=1}^{N_G}(1-p_k)^n\sum_{\ell=0}^\infty  (1-p_k)^{\lfloor c \log (\ell+1)\rfloor}\right)=:p_1(n).
\]
\end{enumerate}
\end{lemma}
\begin{proof}
The proof is almost the same as that of Proposition 4.3 in \cite{D-H-J-K}. Thus, we omit the proof.
\end{proof}

%
%\begin{remark}
%The similar analysis holds for other types of random variables rather than Poission random variable if the following result can be obtained: let $t_{\ell+1} - t_{\ell} = 1 + X_\ell$, where $X_\ell$'s are independent and for some $M\in\bbn$,
%\[\mathbb{P}\left( \omega \ : \ \sum_{\ell=a_{(i-1)(N-1)}}^{a_{i(N-1)}-1} X_\ell(\omega) \ge M(n + \lfloor c \log i(N-1)\rfloor), \quad \mbox{for \textcolor{red}{some} } i \in \bbn\right) \le \tilde{p}(n),  \]
%where $\tilde{p}(n) \to 0$ as $n \to \infty$, but in this case, the condition (1) and (2) in Theorem \ref{T2.1} should be changed to
%\begin{align*}
%&(1')~~\frac{ (N-1+M) \log\left(\frac{1}{1-h\kappa}\right)}{\min\limits_{1\leq k\leq N_G}\log\frac{1}{1-p_k}}<1\\
%&(2')~~
%\frac{1}{\phi(r)}=O(r^{ \varepsilon })\quad\mbox{as}\quad r\to\infty\quad\mbox{for some}\quad0\leq\varepsilon<\frac{1}{N-1}-\frac{(N-1+M)  \log\left(\frac{1}{1-h\kappa}\right)} {(N-1)\min\limits_{1\leq k\leq N_G}\log\frac{1}{1-p_k}},
%\end{align*}
%respectively. This covers the case when $t_{\ell+1}- t_\ell$ can only take the finite number of values in $\bbn$.
%
%\end{remark}

\vspace{0.5cm}

\noindent {\bf Proof of Theorem \ref{T3.1}}: We choose $\varepsilon$ so that
\[ \displaystyle 0 \le \e < \frac{1}{N-1} - \frac{(M+N-1) \log\left(\frac{1}{1-h\kappa}\right)  }{(N-1)\min\limits_{1\le k \le N_G} \log\frac{1}{1-p_k}}, \quad  \mbox{or equivalently}  \]
\[ \frac{1}{\min\limits_{1\le k \le N_G} \log\frac{1}{1-p_k}} <  \frac{1-\e (N-1)}{(M+N-1)  \log\left(\frac{1}{1-h\kappa}\right)}, \]
and we set
\[
c:= \frac{1}{2} \left( \frac{1}{\min\limits_{1\le k \le N_G} \log\frac{1}{1-p_k}} + \frac{1-\e (N-1)}{(M+N-1)   \log\left(\frac{1}{1-h\kappa}\right)}    \right).
\]
First, note that 
\begin{align*}
\begin{aligned}
& c > \frac{1}{\min\limits_{1\le k \le N_G} \log\frac{1}{1-p_k}}, \quad c (M+N-1)    \log\left(\frac{1}{1-h\kappa}\right)    <1, \\
& \mbox{and} \quad  0 \le \e < \frac{1- c (M+N-1)  \log\left(\frac{1}{1-h\kappa}\right)}   {N-1}. 
\end{aligned}
\end{align*}
Now, we choose any $n \in \bbn$ such that

\[
 \sum_{k=1}^{N_G} (1-p_k)^{n}\leq\frac{1}{2}.
 \]
Define the events $A$, $B$ and $\mathcal{F}$ as

\begin{align*}
&A \ : \   \mathcal{G}([t_\ell^*, t_{\ell+1}^*))(\omega) \mbox{ has a spanning tree for any } \ell \ge 0,\\
&B \ : \ \sum_{\ell=a_{(i-1)(N-1)}}^{a_{i(N-1)}-1} T_{\ell}(\omega) < M(n + \lfloor c \log i(N-1)\rfloor), \quad \mbox{for each } i \in \bbn,\\
&\mathcal{F} \ : \ \exists \ x^\infty>0 \mbox{ s.t } \sup_{0\le t < \infty }\D(X[t,\omega]) \le x^\infty \mbox{ and } \lim_{t \to \infty} \D(V[t,\omega]) =0.
\end{align*}
Then, it follows from Proposition \ref{P4.1} that
\[ \mathbb{P}(\mathcal{F} \ | \ A \cap B) = 1.\]
Hence, we use Proposition \ref{P4.2}, Lemma \ref{L4.4} to see
\begin{align*}
\mathbb{P}(\mathcal{F}) &\ge \mathbb{P}(\mathcal{F} \cap (A\cap B)) = \mathbb{P}(A\cap B)  \mathbb{P}(\mathcal{F} \ | \ A \cap B)\\
&= 1- \mathbb{P}(A^c\cup B^c) \ge 1 - (\mathbb{P}(A^c) + \mathbb{P}(B^c)) = \mathbb{P}(A) - \mathbb{P}(B^c)\\
&\ge p_1(n) - p_2(n) \to 1 \quad \mbox{as} \quad n \to \infty.
\end{align*}
This implies our desired result.

\vspace{0.4cm}

\section{The dwelling time process and the continuous CS model} \label{sec:5}
\setcounter{equation}{0}
In this section, we consider two explicit processes for the dwelling time process $T_\ell$ which is equivalent to specifying the switching time process for the change of network topologies, and show that they satisfy one of the key conditions in our framework:
\[
\lim_{n \to \infty} \mathbb{P}\left(  \omega \ : \ \sum_{\ell=a_{(i-1)(N-1)}}^{a_{i(N-1)}-1} T_\ell(\omega) \ge M(n + \lfloor c \log i(N-1)\rfloor), \quad \mbox{for some } i \in \bbn\right ) = 0. 
\]
Moreover, we discuss an improved stochastic flocking estimate compared to authors' recent work \cite{D-H-J-K}. 
\subsection{Dwelling time process} \label{sec:5.1}
In this subsection, we present two explicit random sequence $(T_\ell)_{\ell\geq0}$ satisfying $(\mathcal A 2)$.  \newline

\noindent $\bullet$~{\bf (Poisson's process)}: Suppose that the dwelling time process $T_\ell$ is given by a sequence of Poisson's random variable with a rate $\lambda_\ell$: 
\begin{equation*} \label{F-0}
\mathbb P(N_{\lambda_\ell}=k)=  \frac{(\lambda_\ell)^k}{k !} e^{-\lambda_\ell }, \quad k=0,1,2,\cdots.
\end{equation*}
\begin{proposition}\label{P5.1}
Suppose that the process $T_\ell=N_{\lambda_\ell}$ is a sequence of independent Poisson random variables with the parameters $\lambda_\ell$ satisfying
\[
\lambda_{max}:=\sup_{\ell\geq0}\lambda_\ell<\infty,
\]
and let $(X,V)$ be a solution process to \eqref{CS}. %Define the subsequence $\{t_\ell^*\}_{\ell\geq0}\subset\{t_\ell\}_{\ell\geq0}$ by $t_\ell^*:= t_{a_\ell(n,c)}$.
Then, for any $c>0$, we can choose $M\in\bbn$ so that we have the following estimate:

\[\mathbb{P}\left( \omega \ : \ \sum_{\ell=a_{(i-1)(N-1)}}^{a_{i(N-1)}-1}N_{\lambda_\ell}(\omega) \ge M( n + \lfloor c \log i(N-1)\rfloor), \quad \mbox{for some } i \in \bbn\right) \le p_2(n),  \]
where $\{ p_2(n) \}$ is a sequence satisfying 
\[ \lim_{n \to \infty} p_2(n) = 0. \]
\end{proposition}
\begin{proof}
First, we use $\mathbb{P}\left( \cup_{i=1}^\infty B_i \right) \le \sum_{i=1}^\infty \mathbb{P}(B_i) $ to get
\begin{align*}
&\mathbb{P}\left( \omega \ : \ \sum_{\ell=a_{(i-1)(N-1)}}^{a_{i(N-1)}-1}N_{\lambda_\ell}(\omega) \ge M( n + \lfloor c \log i(N-1)\rfloor), \quad \mbox{for some } i \in \bbn\right)\\
& \hspace{1cm} \leq \sum_{i=1}^\infty \mathbb{P}\left( \omega \ : \ \sum_{\ell=a_{(i-1)(N-1)}}^{a_{i(N-1)}-1}N_{\lambda_\ell}(\omega) \ge M( n + \lfloor c \log i(N-1)\rfloor)\right)\\
& \hspace{1cm}  =\sum_{i=1}^\infty \mathbb{P}\left( \omega \ : \ N_{\tilde{\lambda}_i}(\omega) \ge M( n + \lfloor c \log i(N-1)\rfloor)\right),
\end{align*}
where $\tilde{\lambda}_i = \sum\limits_{\ell=a_{(i-1)(N-1)}}^{a_{i(N-1)}-1}\lambda_\ell $ and we used the independence to get
\[  \sum_{\ell=a_{(i-1)(N-1)}}^{a_{i(N-1)}-1}N_{\lambda_\ell} \sim N_{\tilde{\lambda}_i}. \]
For simplicity, we set
\[ H(i,n) :=n + \lfloor c \log i(N-1)\rfloor, \]
and we have, for $i \ge 1$,
\begin{equation}\label{F-1}
\tilde{\lambda}_i = \sum\limits_{\ell=a_{(i-1)(N-1)}}^{a_{i(N-1)}-1}\lambda_\ell \le \left(a_{i(N-1)} - a_{(i-1)(N-1)}\right)\lambda_{max} \le (N-1)\lambda_{max}H(i,n).
\end{equation}
%where we used
%
%\begin{align*}
%a_{i(N-1)} &= a_{(i-1)(N-1)} + (N-1)n + \sum_{\ell=(i-1)(N-1)}^{i(N-1)-1} \lfloor c \log(\ell+1)\rfloor\\
%&\ge (N-1)\left(n + \left\lfloor c \log\Big( (i-1)(N-1)+1 \Big) \right\rfloor \right).
%\end{align*}
Then, it follows from the distribution of Poisson random variables and \eqref{F-1} that 

\begin{align*}
&\mathbb{P}\left( \omega \ : \ N_{\tilde{\lambda}_i}(\omega) \ge M H(i,n)\right)= \sum_{r= M H(i,n)}^\infty \frac{(\tilde{\lambda}_i)^r e^{-\tilde{\lambda}_i}}{r!}\le \frac{(\tilde{\lambda}_i)^{ M H(i,n)}}{ (M H(i,n))!}\\
&\qquad\le \frac{((N-1)\lambda_{max}H(i,n))^{M H(i,n)}}{(M H(i,n))!}\le \frac{((N-1)e\lambda_{max}/M)^{M H(i,n)}}{\sqrt{2\pi M H(i,n)}},
%&\le \frac{1}{\sqrt{2\pi}} \left(H(i,n)\right)^{-H(i,n)-1/2}  (e(N-1)MH(i,n))^{H(i,n)} \left(\log S(i,n) \right)^{-mH(i,n)}\\
%&= \frac{1}{\sqrt{2\pi}} \left(H(i,n)\right)^{-1/2}(e(N-1)M)^{H(i,n)} \left(\log S(i,n) \right)^{-mH(i,n)},
\end{align*}
where the first inequality follows from the Taylor approximation for $e^{\tilde{\lambda}i}$:
\[
e^{\tilde{\lambda}_i} - \sum_{r=0}^{M H(i,n)-1} \frac{(\tilde{\lambda}_i)^r }{r!} = \frac{(\tilde{\lambda}_i)^{MH(i,n)}}{(MH(i,n))!} \frac{d^{MH(i,n)}}{dx^{M H(i,n)}} e^x \bigg|_{x= \lambda\in(0,\tilde{\lambda}_i)} \le \frac{(\tilde{\lambda}_i)^{M H(i,n)} e^{\tilde{\lambda}_i}}{(MH(i,n))!},
\]
and the last one is from the Stirling's approximation: $n!\geq\sqrt{2\pi}n^{n+\frac{1}{2}}e^{-n}$ for all $n\in\bbn$.\newline

\noindent For $M>(N-1)e\lambda_{max}$, one obtains
\begin{align*}
&\sum_{i=1}^\infty \mathbb{P}\left( \omega \ : \ N_{\tilde{\lambda}_i}(\omega) \ge M( n + \lfloor c \log i(N-1)\rfloor)\right)\\
&\qquad\le \sum_{i=1}^\infty \frac{((N-1)e\lambda_{max}/M)^{MH(i,n)}}{\sqrt{2\pi MH(i,n)}}\\
&\qquad\le \frac{1}{\sqrt{2\pi Mn}}\sum_{i=1}^\infty  ((N-1)e\lambda_{max}/M)^{M(n +  c \log i(N-1)-1)} \\
&\qquad= \frac{((N-1)e\lambda_{max}/M)^{M(n +  c \log (N-1)-1)} }{\sqrt{2\pi Mn}}\sum_{i=1}^\infty  i^{  M c \log ((N-1)e\lambda_{max}/M)}\\
&\qquad=:p_2(n)\rightarrow 0,\quad \textrm{ as }n\rightarrow\infty.
\end{align*}
The last series converges if and only if $M c \log ((N-1)e\lambda_{max}/M)<-1$. This holds for sufficiently large $M>0$, since
\[
\lim_{M\to+\infty}M c \log ((N-1)e\lambda_{max}/M)=-\infty.
\]
\end{proof}

\vspace{0.5cm}

\noindent $\bullet$~{\bf (Geometric process)}: Next, we consider a geometric process for $T_\ell$. Suppose that $T_\ell=Y_{p_\ell}$ is a sequence of independent random variables following geometric distributions:
\begin{equation} \label{F-2}
\mathbb P(Y_{p_\ell}=k)=(1-p_\ell)^kp_\ell,\quad k=0,1,2,\cdots,
\end{equation}
 with the parameters satisfying
\[
p_{min}:=\inf_{\ell\geq0}p_\ell>\frac{1}{2}.
\] 
We first begin with  a technical lemma.
\begin{lemma}\label{L3.4}
For positive integers $A$ and $B$, one has
\[
\frac{\partial}{\partial x_{k}}\left[\sum_{\substack{j_\ell\geq0,~\sum_{\ell=1}^{B} j_\ell\leq A}}\left( \prod_{\ell=1}^{B} (1-x_\ell)^{j_\ell}x_\ell\right)\right]>0,
\]
for $0<x_1,\cdots,x_B<1,~~k=1,\cdots,B.$
\end{lemma}
\begin{proof} Without loss of generality, it suffices to consider $k = 1$ case. Then, one has 
\begin{align*}
& \sum_{\substack{j_\ell\geq0,~\sum_{\ell=1}^{B} j_\ell\leq A}}\left( \prod_{\ell=1}^{B} (1-x_\ell)^{j_\ell}x_\ell\right)= \sum_{\substack{j_\ell\geq0,~\sum_{\ell=2}^{B} j_\ell\leq A}}\sum_{ j_1=0}^{A-\sum_{\ell=2}^{B} j_\ell}\left( \prod_{\ell=1}^{B} (1-x_\ell)^{j_\ell}x_\ell\right)\\
&\hspace{1.5cm} = \sum_{\substack{j_\ell\geq0,~\sum_{\ell=2}^{B} j_\ell\leq A}}\left( \prod_{\ell=2}^{B} (1-x_\ell)^{j_\ell}x_\ell\right)\Big(1-(1-x_1)^{A-\sum_{\ell=2}^{B} j_\ell+1}\Big).
\end{align*}
Hence, we have
\begin{align*}
&\frac{\partial}{\partial x_{k}}\left[\sum_{\substack{j_\ell\geq0,~\sum_\ell j_\ell<A}}\left( \prod_{\ell=1}^{B} (1-x_\ell)^{j_\ell}x_\ell\right)\right]\\
&\hspace{0.5cm} =\sum_{\substack{j_\ell\geq0,~\sum_{\ell=2}^{B} j_\ell\leq A}}\left( \prod_{\ell=2}^{B} (1-x_\ell)^{j_\ell}x_\ell\right)\Big(A-\sum_{\ell=2}^{B} j_\ell+1\Big) (1-x_1)^{A-\sum_{\ell=2}^{B} j_\ell} >0.
\end{align*}
\end{proof}
\begin{proposition}\label{P5.2}
Suppose that $(Y_{p_\ell})$ is a geometric process whose distributions are given by \eqref{F-2}, and let $(X,V)$ be a solution process to \eqref{CS}.  %Define the subsequence $\{t_\ell^*\}_{\ell\geq0}\subset\{t_\ell\}_{\ell\geq0}$ by $t_\ell^*:= t_{a_\ell(n,c)}$. 
Then  for any $c>0$, we can choose $M\in\bbn$ such that 
\[\mathbb{P}\left( \omega \ : \ \sum_{\ell=a_{(i-1)(N-1)}}^{a_{i(N-1)}-1}Y_{p_\ell}(\omega) \ge {M}( n + \lfloor c \log i(N-1)\rfloor), \quad \mbox{for {some} } i \in \bbn\right) \le p_2(n),  \]
where $p_2(n) \to 0$ as $n \to \infty$.
\end{proposition}
\begin{proof}
Note that

\begin{align*}
&\mathbb{P}\left( \omega \ : \ \sum_{\ell=a_{(i-1)(N-1)}}^{a_{i(N-1)}-1} Y_{p_\ell}(\omega) \ge M( n + \lfloor c \log i(N-1)\rfloor), \quad \mbox{for  some } i \in \bbn\right)\\
& \hspace{2.5cm} \leq \sum_{i=1}^\infty \mathbb{P}\left( \omega \ : \ \sum_{\ell=a_{(i-1)(N-1)}}^{a_{i(N-1)}-1}Y_{p_\ell}(\omega) \ge M( n + \lfloor c \log i(N-1)\rfloor)\right).
\end{align*}
For simplicity, we set 

\[ J(i,n) := a_{i(N-1)} - a_{(i-1)(N-1)}, \quad H(i,n) :=n + \lfloor c \log i(N-1)\rfloor. \]
Then, for $i \ge 1$,

\begin{align}\label{C-17}
\begin{aligned}
& \mathbb{P}\left( \omega \ : \ \sum_{\ell=a_{(i-1)(N-1)}}^{a_{i(N-1)}-1}Y_{p_\ell}(\omega) <MH(i,n)\right)  \\
& \hspace{0.5cm} =\sum_{\substack{j_\ell\geq0,~\sum_\ell j_\ell<MH(i,n)}}\left( \prod_{\ell=a_{(i-1)(N-1)}}^{a_{i(N-1)}-1} (1-p_\ell)^{j_\ell}p_\ell\right) \\
& \hspace{0.5cm} \geq\sum_{\substack{j_\ell\geq0,~\sum_\ell j_\ell<MH(i,n)}}\left( \prod_{\ell=a_{(i-1)(N-1)}}^{a_{i(N-1)}-1} (1-p_{min})^{j_\ell}p_{min}\right)  \\
& \hspace{0.5cm} =\sum_{r=0}^{MH(i,n)-1}\sum_{\substack{j_\ell\geq0,~\sum_\ell j_\ell=r}} (1-p_{min})^{r}p_{min}^{J(i,n)} \\
& \hspace{0.5cm} =\sum_{r=0}^{MH(i,n)-1}\binom{r+J(i,n)-1}{J(i,n)-1} (1-p_{min})^{r}p_{min}^{J(i,n)},
\end{aligned}
\end{align} 
where we used Lemma \ref{L3.4} in the inequality. Then, it follows from  \eqref{C-17} that

\begin{align*}
&\mathbb{P}\left( \omega \ : \ \sum_{\ell=a_{(i-1)(N-1)}}^{a_{i(N-1)}-1}Y_{p_\ell}(\omega) \geq MH(i,n)\right)\\
&\hspace{1cm} = p_{min}^{J(i,n)}\left(\frac{1}{p_{min}^{J(i,n)}}-\sum_{r=0}^{MH(i,n)-1}\binom{r+J(i,n)-1}{J(i,n)-1} (1-p_{min})^{r}\right)\\
& \hspace{1cm}  \le \binom{J(i,n) + MH(i,n)-1}{J(i,n)-1} \frac{(1-p_{min})^{MH(i,n)}}{p_{min}^{MH(i,n)}}  \\
& \hspace{1cm} \le \binom{J(i,n) + MH(i,n)}{J(i,n)} \frac{(1-p_{min})^{MH(i,n)}}{p_{min}^{MH(i,n)}}\\
&\hspace{1cm} \le \binom{ M(N-1)H(i,n)}{(N-1)(i,n)} \frac{(1-p_{min})^{MH(i,n)}}{p_{min}^{MH(i,n)}}\\
& \hspace{1cm}  \le \frac{e }{\sqrt{2}\pi   }\left(\Big(1+\frac{N-1}{M}\Big)^{1+\frac{N-1}{M}}\left(\frac{M}{N-1}\right)^{\frac{N-1}{M}}\right)^{MH(i,n)} \frac{(1-p_{min})^{MH(i,n)}}{p_{min}^{MH(i,n)}},
\end{align*}
where we used $J(i,n) \le (N-1)H(i,n)$ and the first inequality follows from the Taylor approximation for $(1-x)^{-J(i,n)}$:

\[
\begin{split}
&\frac{1}{p_{min}^{J(i,n)}}-\sum_{r=0}^{MH(i,n)-1}\binom{r+J(i,n)-1}{r} (1-p_{min})^{r}\\
&\qquad= \frac{(1-p_{min})^{MH(i,n)}}{(MH(i,n))!} \frac{d^{MH(i,n)}}{dx^{MH(i,n)}} \frac{1}{(1-x)^{J(i,n)}} \bigg|_{x= p\in(0,1-p_{min})} \\
&\qquad\le \binom{J(i,n) +MH(i,n)-1}{MH(i,n)} \frac{(1-p_{min})^{MH(i,n)}}{p_{min}^{ J(i,n)+MH(i,n)}},
\end{split}
\]
and the last one is from the following inequality obtained by Stirling's approximation:

\begin{align*}
\binom{n+m}{m} &=\frac{(n+m)!}{n!m!} \leq  \frac{e(n+m)^{n+m+\frac{1}{2}}e^{-n-m}}{\sqrt{2\pi}n^{n+\frac{1}{2}}e^{-n}\cdot\sqrt{2\pi}m^{m+\frac{1}{2}}e^{-m}} =\frac{e(n+m)^{n+m+\frac{1}{2}}}{2\pi n^{n+\frac{1}{2}}m^{m+\frac{1}{2}} }\\
&=\frac{e(n+m)^{n+m }}{2\pi n^{n }m^{m } }\sqrt{\frac{1}{n}+\frac{1}{m}} \leq \frac{e(n+m)^{n+m }}{\sqrt{2}\pi n^{n }m^{m } }\\
&=\frac{e }{\sqrt{2}\pi   }\left(\Big(1+\frac{m}{n}\Big)^{1+\frac{m}{n}}\Big(\frac{n}{m}\Big)^{\frac{m}{n}}\right)^{n},\quad n,m\geq1.
\end{align*} 
Now, we choose $M\in\bbn$ large enough so that

\[
C(M):=\Big(1+\frac{N-1}{M}\Big)^{1+\frac{N-1}{M}}\left(\frac{M}{N-1}\right)^{\frac{N-1}{M}}\frac{ 1-p_{min} }{p_{min} }<1\quad\mbox{and}\quad Mc\log C(M)<-1,
\]
which is possible since $\lim\limits_{M\to\infty}C(M)=\frac{ 1-p_{min} }{p_{min} }<1$. Then, one obtains

\begin{align*}
&\sum_{i=1}^\infty \mathbb{P}\left( \omega \ : \ \sum_{\ell=a_{(i-1)(N-1)}}^{a_{i(N-1)}-1}Y_{p_\ell}(\omega) \geq MH(i,n)\right)\\
& \hspace{1cm} \le \sum_{i=1}^\infty \frac{e }{\sqrt{2}\pi   }\left(C(M)\right)^{MH(i,n)} \le \frac{e }{\sqrt{2}\pi   } \sum_{i=1}^\infty \left(C(M)\right)^{M(n +  c \log i(N-1)-1)}  \\
&\hspace{1cm} \le \frac{e }{\sqrt{2}\pi   } \left(C(M)\right)^{M(n +  c \log (N-1)-1)}\sum_{i=1}^\infty  i^{  M c \log C(M)} := p_2(n) \to 0, \quad \mbox{as} \quad n \to \infty.
\end{align*}
\end{proof}

\subsection{The continuous CS model} \label{sec:5.2}
In this subsection, we use our results for the discrete Cucker-Smale model to extend the previous result \cite{D-H-J-K} for the continuous Cucker-Smale model. In the previous work, the random variables $t_{\ell+1}- t_\ell$ are assumed to be i.i.d with a probability density function $f$ with a compact support. More precisely, the probability density function $f$ is assumed to be supported on a finite interval $[a,b]$ with $0<a<b<\infty$. Based on previous estimates in \cite{D-H-J-K} and this work together with suitable assumptions, we can consider the case when $f$ is supported on $[a,\infty)$.  \newline

In system \eqref{CS}, we adopt the switching law $\sigma : [0,\infty) \times \Omega \to \{1, \cdots, N_G\}$ and the sequence $\{t_{\ell+1}-t_\ell \}_{l\ge 0}$ satisfy the following conditions:
\begin{itemize}
\item
For each $\ell \ge 0$, $t_{\ell+1}-t_\ell$ is given by
\[ t_{\ell+1}-t_\ell = a + T_\ell, \]
where $a>0$ is a constant and $T_\ell$'s are nonnegative, independent random variables.

\vspace{0.2cm}

\item For each $\ell \geq 0$ and $\omega \in \Omega$, $\sigma_t(\omega)$ is constant on the interval $t\in[t_\ell(\omega),t_{\ell+1}(\omega))$.

\vspace{0.2cm}

\item $\{\sigma_{t_{\ell}} \}_{\ell\ge 0}$ is a sequence of i.i.d. random variables such that for any $\ell \geq 0$,
\[
\mathbb P(\sigma_{t_\ell}= k)=p_k,\quad \mbox{ for each } k=1,\cdots,N_G,
\]
where $p_1,\cdots,p_{N_G}>0$ are given positive constants satisfying $p_1+\cdots+p_{N_G}=1$.
\end{itemize}
For each $k=1, \cdots, N_G$, let $\mathcal G_k=(\mathcal V, \mathcal E_k)$ be the $k$-th admissible digraph. Then, for each $t\geq0$ and $\omega\in\Omega$, the time-dependent network topology $(\chi_{ij}^\sigma)=(\chi_{ij}^{\sigma_t(\omega)})$ is determined by

\[\chi_{ij}^{\sigma_t(\omega)} := \left\{\begin{array}{ccc}
1 & \mbox{ if } & (j,i)\in\mathcal E_{\sigma_t(\omega)},\\
0 & \mbox{ if } & (j,i)\notin \mathcal E_{\sigma_t(\omega)}.
\end{array}\right.\]
\vspace{0.1cm}

\noindent For technical reasons and without loss of generality, we assume that each admissible digraph $\mathcal{G}_k$ has a self-loop at each vertex, and we define the union graph of $\mathcal{G}_{\sigma[t,\omega]}$ for $s_0\le t < s_1$ and $\omega\in\Omega$:
\[\mathcal{G}([s_0, s_1))(\omega) := \bigcup_{t=s_0}^{s_1}\mathcal{G}_{\sigma_t(\omega)} =\left(\mathcal{V}, \bigcup_{t= s_0}^{s_1}\mathcal{E}_{\sigma_t(\omega)}\right).\\ \]
Now, we impose the following assumptions $(\mathcal B)$ on the set of admissible digraphs and the distributions of the switching times:
\begin{itemize}
\item $(\mathcal B1)$: The union digraph of the elements of $\mathcal S$ has a spanning tree. In other words,
\[
\bigcup_{1\leq k\leq N_G}\mathcal{G}_k =\left(\mathcal{V}, \bigcup_{1\leq k\leq N_G}\mathcal{E}_{k}\right)\quad\mbox{has a spanning tree.}
\]
\item $(\mathcal B2)$: For some $M>0$, the random variables satisfy
\[\mathbb{P}\left( \omega \ : \ \sum_{\ell=a_{(i-1)(N-1)}}^{a_{i(N-1)}-1} T_\ell(\omega) \ge M(n + \lfloor c \log i(N-1)\rfloor), \quad \mbox{for some } i \in \bbn\right) \le \tilde{p}(n),  \]
where $\tilde{p}(n) \to 0$ as $n \to \infty$.
\end{itemize}
Then, the similar analysis together with the proof in \cite{D-H-J-K} used in previous sections yields the following improved result.
\begin{theorem}\label{T5.1}
Suppose that parameters $N$, $\kappa$, choice probability $p_k$'s and communication weight $\phi$ satisfy 
\[  \frac{ (a(N-1)+M)\kappa}{\min\limits_{1\leq k\leq N_G}\log\frac{1}{1-p_k}}<1, \qquad \frac{1}{\phi(r)}={\mathcal O}(r^{ \varepsilon })\quad\mbox{as}\quad r\to\infty,
\]
where $\varepsilon$ is a positive constant satisfying 
\[ 0\leq\varepsilon<\frac{1}{N-1}-\frac{(a(N-1)+M)\kappa }{(N-1)\min\limits_{1\leq k\leq N_G}\log\frac{1}{1-p_k}}, \]
and let $(X, V)$ be a solution process to \eqref{CS}.  Then, one has the stochastic flocking:
\[\mathbb{P}\Big( \omega : \exists \ x^\infty>0 \mbox{ s.t } \sup_{0\le t < \infty }\D(X(t,\omega) \le x^\infty, \mbox{ and } \lim_{t \to \infty} \D(V(t,\omega)) =0 \Big) = 1. \]
\end{theorem}

\vspace{0.2cm}

\begin{remark}
In the previous work \cite{D-H-J-K}, we assumed that the sequence of dwelling times $\{t_{\ell+1}- t_\ell\}$ are i.i.d with a probability density function $f$ supported on a finite interval $[a,b]$ with $0<a<b<\infty$. This compact support condition was previously assumed to get the upper bound estimate for the sequence $\{t_{r(N-1)}^* - t_{(r-1)(N-1}^*\}_{r\in\bbn}$ which also appears in Proposition \ref{P4.1}. However, thanks to our current analysis, the compact support condition of $f$ can be weakened and we may regard $f$ as being supported on $[a,\infty)$, which is an improvement from the previous paper.
\end{remark}

\vspace{0.2cm}

\section{Conclusion}\label{sec:6}
In this work, we provided a sufficient framework for the stochastic flocking for the discrete Cucker-Smale model with randomly switching topologies. Our sufficient framework involves conditions on the admissible network topologies and switching instants. First, we assume that the union of network topologies in an admissible set contains at least one spanning tree so that all the particles will be eventually connected in the space-time domain. Second, we assume that the dwelling times are short enough in a probabilistic sense so that all the particles communicate sufficiently often in a whole topological-temporal domain. Under these two assumptions, we showed that the discrete Cucker-Smale model exhibits a stochastic flocking for generic initial data asymptotically. For definiteness, we also show that the Poisson process and the geometric process for the dwelling time process both satisfy the second assumption in our framework. Although our current work corresponds to the discrete analog of the continuous one, we can improve the earlier result \cite{D-H-J-K} for the continuous C-S model by removing the compact support assumption on the probability distribution function for switching time process.
There are many other issues that we did not investigate in this work. For example, our analysis focuses only on global flocking. However, as noticed in literature, even for all-to-all couplings, the Cucker-Smale ensemble may not exhibit a global stochastic flocking depending on the decaying nature of the communication weight function, namely, formation of local flocking. Moreover, in this work, we did not consider other physical effects such as time-delay, random media, interaction with other internal variables, etc. These issues will be treated in our future work.

\end{document}